\documentclass[3p,10pt]{elsarticle}

\usepackage{amssymb,latexsym,amsmath,color,subfig,amsthm,epstopdf,siunitx}

\journal{}

\newcommand{\eps}{\varepsilon}
\newcommand{\set}[1]{\left\{#1\right\}}

\newcommand{\kb}{k_\mathrm{b}}
\newcommand{\epsb}{\eps_\mathrm{b}}
\newcommand{\mub}{\mu_\mathrm{b}}
\newcommand{\sigmab}{\sigma_\mathrm{b}}
\newcommand{\p}{\partial}
\newcommand{\E}{\mathrm{E}}
\newcommand{\mA}{\mathbf{A}}
\newcommand{\mE}{\mathbf{E}}
\newcommand{\mH}{\mathbf{H}}

\newcommand{\mU}{\mathbf{G}}
\newcommand{\mV}{\mathbf{F}}
\newcommand{\mW}{\mathbf{W}}
\newcommand{\ma}{\mathbf{a}}
\newcommand{\mr}{\mathbf{r}}

\newcommand{\vt}{\boldsymbol{\theta}}

\DeclareMathOperator*{\inc}{inc}
\DeclareMathOperator*{\scat}{scat}
\DeclareMathOperator*{\tot}{tot}
\DeclareMathOperator*{\kir}{KIR}

\theoremstyle{plain}
\newtheorem{theorem}{Theorem}[section]

\newtheorem{corollary}{Corollary}[section]

\theoremstyle{remark}

\newtheorem{property}{Property}[section]
\newtheorem{example}{Example}[section]

\begin{document}

\begin{frontmatter}



\title{Real-time tracking of moving objects from scattering matrix in real-world microwave imaging}

\author[SCH1,SCH2]{Seong-Ho Son}
\ead{son@sch.ac.kr}
\address[SCH1]{Department of ICT Convergence, Soonchunhyang University, Asan, 31538, Republic of Korea}
\address[SCH2]{Department of Mechanical Engineering, Soonchunhyang University, Asan, 31538, Republic of Korea}
\author[ETRI]{Kwang-Jae Lee}
\ead{reolee0122@etri.re.kr}
\address[ETRI]{Radio Environment \& Monitoring Research Group, Electronics and Telecommunications Research Institute, Daejeon, 34129, Republic of Korea}
\author[KMU]{Won-Kwang Park}
\ead{parkwk@kookmin.ac.kr}
\address[KMU]{Department of Information Security, Cryptology, and Mathematics, Kookmin University, Seoul  02707, Republic of Korea}

\begin{abstract}
The problem of the real-time microwave imaging of small, moving objects from a scattering matrix, whose elements are measured scattering parameters, without diagonal elements is considered herein. An imaging algorithm based on a Kirchhoff migration operated at single frequency is designed, and its mathematical structure is investigated by establishing a relationship with an infinite series of Bessel functions of integer order and antenna configuration. This is based on the application of the Born approximation to the scattering parameters of small objects. The structure explains the reason for the detection of moving objects via a designed imaging function and supplies its some properties. To demonstrate the strengths and weaknesses of the proposed algorithm, various simulations with real-data are conducted.
\end{abstract}


\end{frontmatter}





The real-time tracking of unknown objects using microwaves is an important and interesting inverse scattering problem that arises in fields such as physics, engineering, and military services, and is nowadays highly related to human life \cite{AKKK,BTS,CLFL,F2,MK1,OC,SK}. Most studies have focused on the development of detection algorithms for unknown objects located in a wide area; however, the detection of the movements of small objects or changes in permittivity/conductivity distributions in small or narrow regions has not been sufficiently investigated. Detecting small objects is not an easy problem to solve but can be applied to various real-world problems, such as the diagnoses of cerebral hemorrhages, imaging of crack propagations in walls or bridges, and through-wall imaging. In general, objects exhibit material properties (permittivities and conductivities) that are different from their background media; thus, most studies have focused on retrieving complete information about these properties. Unfortunately, owing to the intrinsic ill-posedness and nonlinearity of the inverse scattering problem, it cannot be successfully and satisfactorily resolved.

To solve this problem, various inversion techniques and corresponding computational environments have been investigated. The most popular and appropriate approaches in real-world applications are based on the Newton-type iteration scheme, which involves retrieving information about the shape, location, and material properties of unknown objects (minimizers); these unknown objects minimize the discrete norm (generally, $\ell^2-$norm) between the measured data in the presence of true and man-made objects. We refer to several remarkable studies \cite{DL,HSM1,IMD,K,RMMP,S1,VXB}.

Although iteration-based techniques have shown their feasibilities, some preceding conditions, such as good initial guesses that are close to the objects,  \textit{a priori} information about unknown objects, appropriate regularization terms significantly dependent on the problem, evaluations of complex Fr{\'e}chet (or domain) derivatives, must be fulfilled to guarantee successful iteration procedures. Furthermore, large computational costs are still incurred, and extensions to multiple objects still prove to be difficult. Hence, Newton-type iteration schemes are not appropriate in designing real-time detection of moving objects.

As alternatives, various non-iterative algorithms have been investigated, e.g., bifocusing method \cite{JBRBTFC,KP4,KPS,SP2}, direct sampling method \cite{IJZ1,IJZ2,P-DSM3,SLP}, MUltiple SIgnal Classification (MUSIC) algorithm \cite{AILP,P-MUSIC1,P-MUSIC6,PKLS}, linear sampling method \cite{ACZ,CHM,HM1,KR}, topological derivatives \cite{B1,GP,LR1,P-TD5}, Kirchhoff and subspace migrations \cite{AGKLS,AGKPS,P-SUB3,P-SUB8}, and orthogonality sampling method \cite{HN,P-OSM1,P-OSM2,P1}. It is worth mentioning that although complete information about unknown objects cannot be retrieved via non-iterative algorithms, such algorithms are fast, effective, and stable in detecting multiple objects without any \textit{a priori} information. So, instead of completely reconstructing the objects, it would be appropriate to design a real-time algorithm for detecting moving objects. Let us emphasize that in order to apply MUSIC or subspace migration for identifying moving objects, a careful threshold of singular values of the scattering matrix is essential. Topological derivative based technique or linear sampling method are very stable and effective non-iterative techniques but additional operations such as solving adjoint problems or nonlinear integral equations are required. Hence, these methods would be inappropriate as real-time detection techniques. Fortunately, bifocusing method, direct and orthogonality sampling methods with multiple sources, and Kirchhoff migration require no additional operations; therefore, we believe that they can be applicable to design a real-time detection algorithm.

Herein, we apply Kirchhoff migration for a real-time tracking of moving small objects from the scattering matrix, whose elements are measured scattering parameters at a fixed frequency. In contrast to the traditional studies in inverse scattering problems, the diagonal elements of a scattering matrix cannot be determined from a microwave machine, i.e., scattering parameters cannot be obtained when the transmitting and receiving antennas are at the same location. Considering such a limitation, an imaging function  based on the Kirchhoff migration for tracking moving objects is designed, and its mathematical structure was rigorously analyzed by establishing a relation with an infinite series of Bessel functions of integer order and antenna configurations. Based on this relation, we can demonstrate that designed tracking algorithm is fast and effective. Moreover, it is possible to guarantee unique determination of moving objects. To illustrate the feasibilities of the designed technique and to avoid committing \textit{inverse crimes}, numerical simulations are performed with experimental data generated by a microwave machine \cite{KLKJS}.

The remainder of this paper is organized as follows. Section \ref{sec:2} briefly introduces the forward problem and scattering parameters caused by the presence of objects. In Section \ref{sec:3}, we describe the Kirchhoff migration-based imaging algorithm without the diagonal elements of the scattering matrix, analyze the mathematical structure of the imaging function by establishing an infinite series of Bessel functions of integer order and antenna setting, and discuss some properties of the imaging function. Section \ref{sec:4} discusses the results of the simulation, which used real-world data to demonstrate the effectiveness of the algorithm. Section \ref{sec:5} concludes the paper.

\section{The forward problem and scattering parameters}\label{sec:2}
In this section, the forward model and scattering parameters are briefly introduced in the presence of a set of objects with small diameter. Let $D_m(t)$, $m=1,2,\cdots,M$, be an (cross-sectional) object with location $\mr_m(t)$ at time $t$ and $D(t)$ denotes the collection of $D_m(t)$. Throughout this study, all $D_m(t)$ are included in a homogeneous region of interest (ROI) $\Omega\subset\mathbb{R}^2$ and surrounded by several transmitting and receiving antennas $\mA_n$ located at $\ma_n$, $n=1,2,\cdots,N$, with $|\ma_n|=R$. We denote $\mathcal{A}$ as the collection of antennas $\mA_n$ and assume that all $D_m(t)$ and $\Omega$ are characterized by their dielectric permittivity and electric conductivity at a given angular frequency $\omega=2\pi f$, i.e., the magnetic permeability of all objects are constant such that $\mu(\mr,t)\equiv\mub=4\pi\times\SI{e-7}{\henry/\m}$, $\mr\in\Omega$. We denote $\eps(\mr,t)$ and $\sigma(\mr,t)$ as the piecewise constant permittivity and conductivity
\[\eps(\mr,t)=\left\{\begin{array}{rcl}
                     \eps_m & \mbox{for} & \mr\in D_m(t),\\
                     \epsb & \mbox{for} & \mr\in\Omega\backslash\overline{D(t)},
                   \end{array}
\right.
\quad\mbox{and}\quad
\sigma(\mr,t)=\left\{\begin{array}{rcl}
                     \sigma_m & \mbox{for} & \mr\in D_m(t),\\
                     \sigmab & \mbox{for} & \mr\in\Omega\backslash\overline{D(t)},
                   \end{array}
\right.\]
respectively. With this, we further assume that the following relationships hold for $m=1,2,\cdots,M$,
\begin{equation}\label{Relationship}
\omega\epsb\gg\sigmab\quad\text{and}\quad\left(\sqrt{\frac{\eps_m}{\epsb}}-1\right)\text{diam}(D_m)<\frac{\text{wavelength}}{4},
\end{equation}
where $\text{diam}(D_m)$ denotes the diameter of $D_m$.

Let $\mE_{\inc}(\ma_q,\mr)$ be the time-harmonic incident electric field in a homogeneous medium resulting from a point current density at $\mA_q$. Then, based on the Maxwell equation, $\mE_{\inc}(\ma_q,\mr)$ satisfies
\[\nabla\times\mE_{\inc}(\ma_q,\mr)=i\omega\mub\mH_{\inc}(\ma_q,\mr)\quad\mbox{and}\quad\nabla\times\mH_{\inc}(\ma_q,\mr)=(\sigmab-i\omega\epsb)\mE_{\inc}(\ma_q,\mr).\]
Let $\mE_{\tot}(\mr,\ma_p)$ be the corresponding total electric field due to the existence of $D(t)$ measured at $\mA_p$. Then, $\mE_{\tot}(\mr,\ma_p)$ satisfies
\[\nabla\times\mE_{\tot}(\mr,\ma_p)=i\omega\mub\mH_{\tot}(\mr,\ma_p)\quad\mbox{and}\quad\nabla\times\mH_{\tot}(\mr,\ma_p)=(\sigma(\mr,t)-i\omega\eps(\mr,t))\mE_{\tot}(\mr,\ma_p)\]
with the transmission condition on $\p D_m$, $m=1,2,\cdots,M$. Here, the time-harmonic dependence $e^{-i\omega t}$ is assumed.

The scattering parameter ($S-$parameter) $S(p,q,t)$, $p,q=1,2,\cdots,N$, is defined as the ratio of the output voltage at the $\mA_p$ and the input voltage at the $\mA_q$ at the time $t$. Let $S_{\tot}(p,q,t)$ and $S_{\inc}(p,q,t)$ be the scattering parameter data in the presence and absence of $D(t)$, respectively, and denote $S_{\scat}(p,q,t)=S_{\tot}(p,q,t)-S_{\inc}(p,q,t)$ as the scattered field $S$-parameter. Based on the simulation setup introduced in recent works \cite{KLKJS,P-SUB16,P-SUB18}, only the $z$-components of $\mE_{\inc}(\ma_q,\mr)$ and $\mE_{\tot}(\mr,\ma_p)$ can be handled so that by denoting $\E_{\inc}^{(z)}(\ma_q,\mr)$ and $\E_{\tot}^{(z)}(\mr,\ma_p)$ as the $z$-component of incident and total fields, respectively, $S_{\scat}(p,q,t)$ can be presented as the following integral equation formula:
\begin{equation}\label{Sparameter}
S_{\scat}(p,q,t)=\frac{ik_0^2}{4\omega\mub}\int_{\Omega}\left(\frac{\eps(\mr',t)-\epsb}{\epsb}+i\frac{\sigma(\mr',t)-\sigmab}{\omega\epsb}\right)\E_{\inc}^{(z)}(\ma_q,\mr')\E_{\tot}^{(z)}(\mr',\ma_p)d\mr',
\end{equation}
where $k_0$ denotes the lossless background wavenumber that satisfies $k_0^2=\omega^2\epsb\mub$, refer to \cite{HSM2}.

\section{Kirchhoff migration for a real-time detection: introduction, analysis, and some properties}\label{sec:3}
\subsection{Introduction to imaging function of the Kirchhoff migration}
Herein, we introduce the traditional imaging function of the Kirchhoff migration and apply it for a real-time mornitoring of moving objects $D_m(t)$ generated from the scattering matrix $\mathbb{K}(t)$ such that
\begin{equation}\label{ScatteringMatrix}
\mathbb{K}(t)=\begin{bmatrix}
                  S_{\scat}(1,1,t) & S_{\scat}(1,2,t) & \cdots & S_{\scat}(1,N,t) \\
                  S_{\scat}(2,1,t) & S_{\scat}(2,2,t) & \cdots & S_{\scat}(2,N,t) \\
                  \vdots & \vdots & \ddots & \vdots \\
                 S_{\scat}(N,1,t) & S_{\scat}(N,2,t) & \cdots & S_{\scat}(N,N,t)
               \end{bmatrix}.
\end{equation}

Unfortunately, it is very difficult to apply \eqref{Sparameter} directly to design an imaging function because exact expression of $\E_{\tot}^{(z)}(\mr,\ma_p)$ is unknown. Since we already assumed the condition \eqref{Relationship} holds, every $D_m(t)$ can be regarded as an object with small diameter. Thus, based on \cite{SKL}, it is possible to apply the Born approximation $\E_{\tot}^{(z)}(\ma_p,\mr')\approx\E_{\inc}^{(z)}(\ma_p,\mr')$ so that \eqref{Sparameter} can be approximated as
\begin{equation}\label{SparameterIntegral}
S_{\scat}(p,q,t)\approx\frac{ik_0^2}{4\omega\mub}\int_{D(t)}\left(\frac{\eps(\mr',t)-\epsb}{\epsb}+i\frac{\sigma(\mr',t)-\sigmab}{\omega\epsb}\right)\E_{\inc}^{(z)}(\ma_q,\mr')\E_{\inc}^{(z)}(\ma_p,\mr')d\mr'.
\end{equation}
and correspondingly, applying \eqref{SparameterIntegral} to \eqref{ScatteringMatrix}, $\mathbb{K}(t)$ can be decomposed as
\begin{align}
\begin{aligned}\label{Decomposition}
\mathbb{K}(t)&\approx\frac{ik_0^2}{4\omega\mub}\int_{D(t)}\left(\frac{\eps(\mr',t)-\epsb}{\epsb}+i\frac{\sigma(\mr',t)-\sigmab}{\omega\epsb}\right)\begin{bmatrix}
\smallskip\E_{\inc}^{(z)}(\ma_1,\mr')\\
\E_{\inc}^{(z)}(\ma_2,\mr')\\
\vdots\\
\E_{\inc}^{(z)}(\ma_N,\mr')
\end{bmatrix}
\begin{bmatrix}
\E_{\inc}^{(z)}(\ma_1,\mr')&
\cdots&
\E_{\inc}^{(z)}(\ma_N,\mr')
\end{bmatrix}d\mr'\\
&:=\frac{ik_0^2}{4\omega\mub}\int_{D(t)}\mathcal{O}(\mr',t)\mU(\mr',t)\mU(\mr',t)^{\mathtt{T}}d\mr'.
\end{aligned}
\end{align}
Based on the structure of the above decomposition, let us generate the following unit vector: for each $\mr\in\Omega$,
\[\mV(\mr)=\frac{\mW(\mr)}{|\mW(\mr)|},\quad\text{where}\quad\mW(\mr)=\begin{bmatrix}
\smallskip\E_{\inc}^{(z)}(\ma_1,\mr)\\
\E_{\inc}^{(z)}(\ma_2,\mr)\\
\vdots\\
\E_{\inc}^{(z)}(\ma_N,\mr)
\end{bmatrix}.\]

Now, let us define an inner product $\langle\cdot,\cdot\rangle_{\ell^2(\mathcal{A})}$ on the Lebesgue space $\ell^2(\mathcal{A})$ such that
\begin{equation}\label{InnerProduct}
\langle\mV(\mr),\mU(\mr',t)\rangle_{\ell^2(\mathcal{A})}=\overline{\mV(\mr)}^{\mathtt{T}}\mU(\mr',t)=\frac{1}{|\mW(\mr)|}\sum_{n=1}^{N}\overline{\E_{\inc}^{(z)}(\ma_n,\mr)}\E_{\inc}^{(z)}(\ma_n,\mr').
\end{equation}
Then, based on the orthogonal property of the $\ell^2(\mathcal{A})$, we can examine that the value of $\langle\mV(\mr),\mU(\mr',t)\rangle_{\ell^2(\mathcal{A})}$ will reach its maximum value when $\mr=\mr'\in D(t)$. Correspondingly, the following imaging function of the Kirchhoff migration can be introduced: for each $\mr\in\Omega$,
\begin{equation}\label{ImagingFunctionTraditional}
\mathfrak{F}_{\kir}(\mr,t)=|\overline{\mV(\mr)}^{\mathtt{T}}\mathbb{K}(t)\overline{\mV(\mr)}|.
\end{equation}
Based on the \eqref{Decomposition} and \eqref{InnerProduct}, the value of $\mathfrak{F}_{\kir}(\mr,t)$ is expected to reach its maximum value when $\mr=\mr'\in D(t)$. For a detailed description of the imaging function, we refer to \cite{AGKPS}.

Contrary to the traditional simulation setup, each $N$ antenna is used for signal transmission, whereas the remaining $N-1$ antennas are used for signal reception. Therefore, the value of $S_{\scat}(n,n)$ for $n=1,2,\cdots,N$, i.e., the diagonal elements of $\mathbb{K}(t)$ cannot be determined so that generated scattering matrix is of the following form
\[\mathbb{K}(t)=\begin{bmatrix}
\text{unknown} & S_{\scat}(1,2,t) & \cdots & S_{\scat}(1,N,t) \\
S_{\scat}(2,1,t) & \text{unknown} & \cdots & S_{\scat}(2,N,t) \\
\vdots & \vdots & \ddots & \vdots \\
S_{\scat}(N,1,t) & S_{\scat}(N,2,t) & \cdots & \text{unknown}
\end{bmatrix}.\]
For a related discussions, we recommend some references \cite{P-SUB11,P-MUSIC6,SSKLJ}. In this study, we set $S_{\scat}(p,q,t)=0$ instead of the unknown measurement data and the following scattering matrix is considered for the design of the imaging function:
\begin{equation}\label{MSRWithout}
\mathbb{G}(t)=\begin{bmatrix}
                  0 & S_{\scat}(1,2,t) & \cdots & S_{\scat}(1,N,t) \\
                  S_{\scat}(2,1,t) & 0 & \cdots & S_{\scat}(2,N,t) \\
                  \vdots & \vdots & \ddots & \vdots \\
                 S_{\scat}(N,1,t) & S_{\scat}(N,2,t) & \cdots & 0
               \end{bmatrix}.
\end{equation}
We refer to \cite{P-SUB16,P-OSM1,P-KIR1} for an explanation of why the diagonal elements of scattering matrix were set to zero. Although, the diagonal elements are missing, the following imaging function can be introduced with a similar way to \eqref{ImagingFunctionTraditional}: for each $\mr\in\Omega$,
\begin{equation}\label{ImagingFunction}
\mathfrak{F}(\mr,t)=|\overline{\mV(\mr)}^{\mathtt{T}}\mathbb{G}(t)\overline{\mV(\mr)}|.
\end{equation}
Theoretical reason of the applicability of real-time object detection is discussed next.

\subsection{Structure of the imaging function}
In order to explain the availability of real-time detection of moving object, mathematical structure of the designed imaging function is carefully explored by establishing a relationship with an infinite series of Bessel functions.

\begin{theorem}[Structure of imaging function]\label{TheoremStructure}
Let $\vt_n=\ma_n/R=(\cos\theta_n,\sin\theta_n)^{\mathtt{T}}$ and $\mr-\mr'=|\mr-\mr'|(\cos\phi,\sin\phi)^{\mathtt{T}}$. If $|\kb(\ma_n-\mr)|\geq0.25$ for all $\mr\in\Omega$, then $\mathfrak{F}(\mr,t)$ can be represented as follows:
\begin{multline}\label{Structure}
\mathfrak{F}(\mr,t)\approx\Bigg|\int_{D(t)}\mathcal{C}(\mr',t)\left(J_0(\kb|\mr-\mr'|)+\frac{1}{N}\sum_{n=1}^{N}\mathcal{E}(\kb|\mr-\mr'|,t)\right)^2d\mr'\\
-\frac{1}{N}\int_{D(t)}\mathcal{C}(\mr',t)\left(J_0(2\kb|\mr-\mr'|)+\frac{1}{N}\sum_{n=1}^{N}\mathcal{E}(2\kb|\mr-\mr'|,t)\right)d\mr'\Bigg|,
\end{multline}
where $J_s$ is a Bessel function of the order $s$ of the first kind,
\[\mathcal{C}(\mr',t)=\frac{N\omega\epsb}{32\kb R\pi}\left(\frac{\eps(\mr',t)-\epsb}{\epsb}+i\frac{\sigma(\mr',t)-\sigmab}{\omega\epsb}\right),\quad\text{and}\quad\mathcal{E}(\kb|\mr-\mr'|,t)=\sum_{s=-\infty,s\ne0}^{\infty}i^s J_{s}(\kb|\mr-\mr'|)e^{is(\theta_n-\phi)}.\]
\end{theorem}
\begin{proof}
Since $|\kb(\ma_n-\mr)|\geq0.25$ for all $n=1,2,\cdots,N$, the following asymptotic form of the Hankel function holds (see \cite[Theorem 2.5]{CK}, for instance)
\begin{equation}\label{AsymptoticHankel}
H_0^{(1)}(\kb|\mr-\mr'|)\approx\frac{(1-i)e^{i\kb|\mr|}}{\sqrt{\kb\pi|\ma_n|}}e^{-i\kb\vt_n\cdot\mr'}.
\end{equation}
Then, since
\[\mV(\mr)\approx\frac{1}{\sqrt{N}}\bigg[e^{-i\kb\vt_1\cdot\mr},e^{-i\kb\vt_2\cdot\mr},\ldots,e^{-i\kb\vt_N\cdot\mr}\bigg]^{\mathtt{T}}\]
and
\[\mathbb{G}(t)\approx C\begin{bmatrix}
\medskip0&\displaystyle\int_{D(t)}\mathcal{O}(\mr',t)e^{-i\kb(\vt_1+\vt_2)\cdot\mr'}d\mr'&\cdots&\displaystyle\int_{D(t)}\mathcal{O}(\mr',t)e^{-i\kb(\vt_1+\vt_N)\cdot\mr'}d\mr'\\
\displaystyle\int_{D(t)}\mathcal{O}(\mr',t)e^{-i\kb(\vt_2+\vt_1)\cdot\mr'}d\mr'&0&\cdots&\displaystyle\int_{D(t)}\mathcal{O}(\mr',t)e^{-i\kb(\vt_2+\vt_N)\cdot\mr'}d\mr'\\
\vdots&\vdots&\ddots&\vdots\\
\displaystyle\int_{D(t)}\mathcal{O}(\mr',t)e^{-i\kb(\vt_N+\vt_1)\cdot\mr'}d\mr'&\displaystyle\int_{D(t)}\mathcal{O}(\mr',t)e^{-i\kb(\vt_N+\vt_2)\cdot\mr'}d\mr'&\cdots&0\\
\end{bmatrix},\]
we can derive
\[\overline{\mV(\mr)}^{\mathtt{T}}\mathbb{G}(t)\approx\frac{C}{\sqrt{N}}\begin{bmatrix}
\smallskip \displaystyle\int_{D(t)}\mathcal{O}(\mr',t)e^{-i\kb\vt_1\cdot\mr'}\left(\sum_{n\in\mathcal{N}_1}e^{i\kb\vt_n\cdot(\mr-\mr')}\right)d\mr'\\
\displaystyle\int_{D(t)}\mathcal{O}(\mr',t)e^{-i\kb\vt_2\cdot\mr'}\left(\sum_{n\in\mathcal{N}_2}e^{i\kb\vt_n\cdot(\mr-\mr')}\right)d\mr'\\
\vdots\\
\displaystyle\int_{D(t)}\mathcal{O}(\mr',t)e^{-i\kb\vt_N\cdot\mr'}\left(\sum_{n\in\mathcal{N}_N}e^{i\kb\vt_n\cdot(\mr-\mr')}\right)d\mr'
\end{bmatrix},\]
where $\mathcal{N}_p=\set{1,2,\cdots,N}\backslash\set{p}$ and $C=(e^{2iR\kb}\omega\epsb)/(32R\kb\pi)$.

Since $\vt_n\cdot(\mr-\mr')=|\mr-\mr'|\cos(\theta_n-\phi)$ and the following Jacobi-Anger expansion formula holds uniformly
\begin{equation}\label{JacobiAnger}
e^{ix\cos\theta}=J_0(x)+\sum_{s=-\infty,s\ne0}^{\infty}i^s J_{s}(x)e^{is\theta},
\end{equation}
we can evaluate
\begin{align*}
\sum_{n=1}^{N}e^{i\kb\vt_n\cdot(\mr-\mr')}&=\sum_{n=1}^{N}\left(J_0(\kb|\mr-\mr'|)+\sum_{s=-\infty,s\ne0}^{\infty}i^s J_{s}(\kb|\mr-\mr'|)e^{is(\theta_n-\phi)}\right)\\
&=NJ_0(\kb|\mr-\mr'|)+\sum_{n=1}^{N}\mathcal{E}(\kb|\mr-\mr'|,t)
\end{align*}
and correspondingly,
\begin{multline*}
\overline{\mV(\mr)}^{\mathtt{T}}\mathbb{G}(t)\overline{\mV(\mr)}\approx\frac{C}{N}\begin{bmatrix}
\smallskip \displaystyle\int_{D(t)}\mathcal{O}(\mr',t)e^{-i\kb\vt_1\cdot\mr'}\left(\sum_{n=1}^{N}e^{i\kb\vt_n\cdot(\mr-\mr')}-e^{i\kb\vt_1\cdot(\mr-\mr')}\right)d\mr'\\
\displaystyle\int_{D(t)}\mathcal{O}(\mr',t)e^{-i\kb\vt_2\cdot\mr'}\left(\sum_{n=1}^{N}e^{i\kb\vt_n\cdot(\mr-\mr')}-e^{i\kb\vt_2\cdot(\mr-\mr')}\right)d\mr'\\
\vdots\\
\displaystyle\int_{D(t)}\mathcal{O}(\mr',t)e^{-i\kb\vt_N\cdot\mr'}\left(\sum_{n=1}^{N}e^{i\kb\vt_n\cdot(\mr-\mr')}-e^{i\kb\vt_N\cdot(\mr-\mr')}\right)d\mr'
\end{bmatrix}
\begin{bmatrix}
\smallskip e^{i\kb\vt_1\cdot\mr}\\
e^{i\kb\vt_2\cdot\mr}\\
\vdots\\
e^{i\kb\vt_N\cdot\mr}
\end{bmatrix}\\
=\frac{C}{N}\begin{bmatrix}
\smallskip \displaystyle\int_{D(t)}\mathcal{O}(\mr',t)e^{-i\kb\vt_1\cdot\mr'}\left(NJ_0(\kb|\mr-\mr'|)+\sum_{n=1}^{N}\mathcal{E}(\kb|\mr-\mr'|,t)-e^{i\kb\vt_1\cdot(\mr-\mr')}\right)d\mr'\\
\displaystyle\int_{D(t)}\mathcal{O}(\mr',t)e^{-i\kb\vt_2\cdot\mr'}\left(NJ_0(\kb|\mr-\mr'|)+\sum_{n=1}^{N}\mathcal{E}(\kb|\mr-\mr'|,t)-e^{i\kb\vt_2\cdot(\mr-\mr')}\right)d\mr'\\
\vdots\\
\displaystyle\int_{D(t)}\mathcal{O}(\mr',t)e^{-i\kb\vt_N\cdot\mr'}\left(NJ_0(\kb|\mr-\mr'|)+\sum_{n=1}^{N}\mathcal{E}(\kb|\mr-\mr'|,t)-e^{i\kb\vt_N\cdot(\mr-\mr')}\right)d\mr'
\end{bmatrix}
\begin{bmatrix}
\smallskip e^{i\kb\vt_1\cdot\mr}\\
e^{i\kb\vt_2\cdot\mr}\\
\vdots\\
e^{i\kb\vt_N\cdot\mr}
\end{bmatrix}.
\end{multline*}

Now, applying \eqref{JacobiAnger} again, we can evaluate
\[\sum_{n'=1}^{N}e^{i\kb\vt_{n'}\cdot(\mr-\mr')}\left(NJ_0(\kb|\mr-\mr'|)+\sum_{n=1}^{N}\mathcal{E}(\kb|\mr-\mr'|,t)\right)=\left(NJ_0(\kb|\mr-\mr'|)+\sum_{n=1}^{N}\mathcal{E}(\kb|\mr-\mr'|,t)\right)^2\]
and
\[\sum_{n'=1}^{N}e^{2i\kb\vt_{n'}\cdot(\mr-\mr')}=NJ_0(2\kb|\mr-\mr'|)+\sum_{n=1}^{N}\mathcal{E}(2\kb|\mr-\mr'|,t).\]
Then,
\begin{align*}
\overline{\mV(\mr)}^{\mathtt{T}}\mathbb{G}(t)\overline{\mV(\mr)}\approx&\frac{C}{N}\int_{D(t)}\mathcal{O}(\mr',t)\sum_{n'=1}^{N}e^{i\kb\vt_{n'}\cdot(\mr-\mr')}\left(NJ_0(\kb|\mr-\mr'|)+\sum_{n=1}^{N}\mathcal{E}(\kb|\mr-\mr'|,t)\right)d\mr'\\
&-\frac{C}{N}\int_{D(t)}\mathcal{O}(\mr',t)\sum_{n'=1}^{N}e^{2i\kb\vt_{n'}\cdot(\mr-\mr')}d\mr'\\
=&\int_{D(t)}\mathcal{C}(\mr',t)\left(J_0(\kb|\mr-\mr'|)+\frac{1}{N}\sum_{n=1}^{N}\mathcal{E}(\kb|\mr-\mr'|,t)\right)^2d\mr'\\
&-\frac{1}{N}\int_{D(t)}\mathcal{C}(\mr',t)\left(J_0(2\kb|\mr-\mr'|)+\frac{1}{N}\sum_{n=1}^{N}\mathcal{E}(2\kb|\mr-\mr'|,t)\right)d\mr'.
\end{align*}
Finally, by taking the absolute value, we can obtain the result \eqref{Structure}. This completes the proof.
\end{proof}

\subsection{Various properties of imaging function}\label{sec:3.2}
On the basis of the result \eqref{Structure}, we can explore some properties of imaging function as follows.
\begin{property}[Availability of detection]\label{Property1}
Since $J_0(0)=1$ and $J_s(0)=0$ for $s\ne0$, map of $\mathfrak{F}(\mr,t)$ will contain peaks of large magnitudes when $\mr=\mr'\in D(t)$. Hence, the locations of the moving objects can be imaged via the map of $\mathfrak{F}(\mr,t)$. Furthermore, owing to the oscillating property of the Bessel functions, some artifacts will be included on the map of $\mathfrak{F}(\mr,t)$.
\end{property}

\begin{property}[Dependence of the material properties]\label{Property2}
If $\mr=\mr'\in D_m(t)$ then since $J_0(\kb|\mr-\mr'|)=1$ and $\mathcal{E}(\kb|\mr-\mr'|,t)=\mathcal{E}(2\kb|\mr-\mr'|,t)=0$, we have
\[\mathfrak{F}(\mr,t)\approx\bigg|\int_{D_m(t)}\mathcal{C}(\mr',t)d\mr'-\frac{1}{N}\int_{D_m(t)}\mathcal{C}(\mr',t)d\mr'\bigg|=\frac{(N-1)\omega\epsb}{32|\kb|R\pi}\bigg|\frac{\eps_m-\epsb}{\epsb}+i\frac{\sigma_m-\sigmab}{\omega\epsb}\bigg|\text{area}(D_m(t)).\]
Hence, we can conclude that the value of $\mathfrak{F}(\mr,t)$ significantly depends on the size, permittivity, and conductivity of the object. This means that if the size, permittivity, and conductivity of an object $D_m(t)$ is considerably larger than those of the others, peaks of large magnitude will appear on the map of $\mathfrak{F}(\mr,t)$ at the location of this object because $|\mathcal{C}(\mr,t)|>|\mathcal{C}(\mr',t)|$ for $\mr\in D_m(t)$ and $\mr'\in D(t)\backslash\overline{D_m(t)}$.\end{property}

\begin{property}[Ideal conditions for a proper detection]\label{Property3}
Based on the structure of the factor $\mathcal{E}(\kb|\mr-\mr'|,t)$, it disturbs the detection of objects. Thus, eliminating the factors $\mathcal{E}(\kb|\mr-\mr'|,t)$ and $\mathcal{E}(2\kb|\mr-\mr'|,t)$ will guarantee good results. Notice that if one can increase $N$ as much as possible ($N\longrightarrow+\infty$) or apply extremely high frequency ($\omega\longrightarrow+\infty$), then their effects can be reduced but this is inappropriate for real-world applications.\end{property}

\begin{property}[Practical condition for a proper detection]\label{Property4}
To reduce the adverse effect of $\mathcal{E}(\kb|\mr-\mr'|,t)$, a condition must be found to satisfy the following equation:
\[\sum_{n=1}^{N}\sum_{s=-\infty,s\ne0}^{\infty}i^s J_{s}(\kb|\mr-\mr'|)e^{is(\theta_n-\phi)}=0.\]
Since we have no a priori information of objects, $i^sJ_{s}(\kb|\mr-\mr'|)$ cannot be eliminated. This means that we must find a condition $\theta_n$ such that
\[\sum_{s=-\infty,s\ne0}^{\infty}\sum_{n=1}^{N}e^{is(\theta_n-\phi)}=\sum_{s=-\infty,s\ne0}^{\infty}\sum_{n=1}^{N}\cos\big(s(\theta_n-\phi)\big)+i\sin\big(s(\theta_n-\phi)\big)=0.\]
Based on the periodic property of the cosine and sine functions, the effect of $\mathcal{E}(\kb|\mr-\mr'|,t)$ can be reduced when even number of antennas are uniformly distributed on a circular array and total number of antennas $N$ is greater than $8$. This means that the array configuration of the antennas affects the imaging quality and this is the theoretical reason for the even number of antennas and their symmetric location with respect to the origin in general. We refer to \cite{KCP1,P-TD5,P-SUB16} for a similar phenomenon in various imaging techniques.
\end{property}

\begin{property}[On the imaging of objects close to an antenna]\label{Property5}
On the imaging of objects, the distance between each objects or between object and antennas significantly influences the imaging performance because the integral equation formula \eqref{Sparameter} holds for well-separated objects and Theorem \ref{TheoremStructure} holds when $|\kb(\ma_n-\mr)|\geq0.25$. If an object $D_m(t)$ is very close to an antenna $\mA_n$ at time $t$ such that
\[0<|\ma_n-\mr_m(t)|\ll\frac{1}{|\kb|}\]
then based on the asymptotic form of the Hankel function
\[H_0^{(1)}(\kb|\ma_n-\mr'|)\approx 1+\frac{2i}{\pi}\left\{\ln\left(\frac{|\kb(\ma_n-\mr')|}{2}\right)+\gamma\right\}\longrightarrow\text{blow up}\quad\text{as}\quad\mr'\longrightarrow\ma_n,\]
some elements of $\mathbb{G}(t)$ will be very large valued so that unexpected imaging results (appearance of several artifacts with large magnitudes, invisible of objects, etc.) will be obtained. Here, $\gamma=0.57721\ldots$ denotes the Euler–Mascheroni constant.
\end{property}

Finally, based on the Property \ref{Property1}, we can derive following important result of the unique determination.

\begin{corollary}[Unique determination of moving objects]
For given angular frequency $\omega$, the moving objects $D_m(t)$ at each time $t$ can be detected uniquely through the map of $\mathfrak{F}(\mr,t)$ with the same condition of Theorem \ref{TheoremStructure}.
\end{corollary}

\section{Simulation results with experimental data}\label{sec:4}
In this section, various numerical simulation results with  experimental data are shown for demonstrating the feasibility of the Kirchhoff migration and supporting the theoretical result. For the simulation, $N=16$ dipole antennas equally distributed on a circle with a diameter of $\SI{0.18}{\m}$ are placed into a cylindrical tank with a height of $\SI{0.3}{\m}$ and a diameter of $\SI{0.2}{\m}$ to satisfy the condition in Property \ref{Property4}. The tank was filled with water such that $(\epsb,\sigmab)=(78\eps_0,\SI{0.2}{\siemens/\m})$ at $f=\SI{925}{\MHz}$, where $\eps_0=\SI{8.854e-12}{\farad/\meter}$ denotes the vacuum permittivity. Under the current simulation configuration, since $\kb\approx171.27- 4.26i$ (i.e., $|\kb|\approx171.3237$), we set the ROI as a circle centered at the origin with radius of $\SI{0.085}{\m}$ to satisfy the condition $|\kb(\ma_n-\mr)|\geq0.25$ for all $n$, refer to the condition of Theorem \ref{TheoremStructure} and discussion of Property \ref{Property5}. Throughout this paper, the elements $S_{\scat}(p,q,t)$ of $\mathbb{G}(t)$ are generated using a microwave machine manufactured by Electronics and Telecommunications Research Institute (ETRI), refer to \cite{KLKJS}. For describing objects, the cross-section of four long objects $D_m$, $m=1,2,3,4$, are chosen. Table \ref{Materials} presents the material properties of each object and Figure \ref{Configuration} exhibits the manufactured microwave machine, antenna arrangements, and selected objects.

\begin{table}[h]
\begin{center}
\begin{tabular}{c||c|c|c}
\hline \centering Object (cross-section)&~Permittivity~&Conductivity ($\SI{}{\siemens/\m}$)&Diameter ($\SI{}{\mm}$)\\
\hline\hline \centering $D_1(t)$: plastic bar&$3.0\eps_0$ (approximately)&$0$ (approximately)&$20.00$\\
\hline \centering $D_2(t)$: steel bar\phantom{pi}&$-$&$\infty$&$\phantom{2}6.40$\\
\hline \centering $D_3(t)$: steel bar\phantom{pi}&$-$&$\infty$&$\phantom{2}6.55$\\
\hline \centering $D_4(t)$: plastic bar&$2.5\eps_0$ (approximately)&$0$ (approximately)&$\phantom{2}6.40$\\
\hline
\end{tabular}
\centering\caption{\label{Materials}Values of permittivities, conductivities, and sizes of objects.}
\end{center}
\end{table}

\begin{figure}[h]
\begin{center}
\includegraphics[width=.99\textwidth]{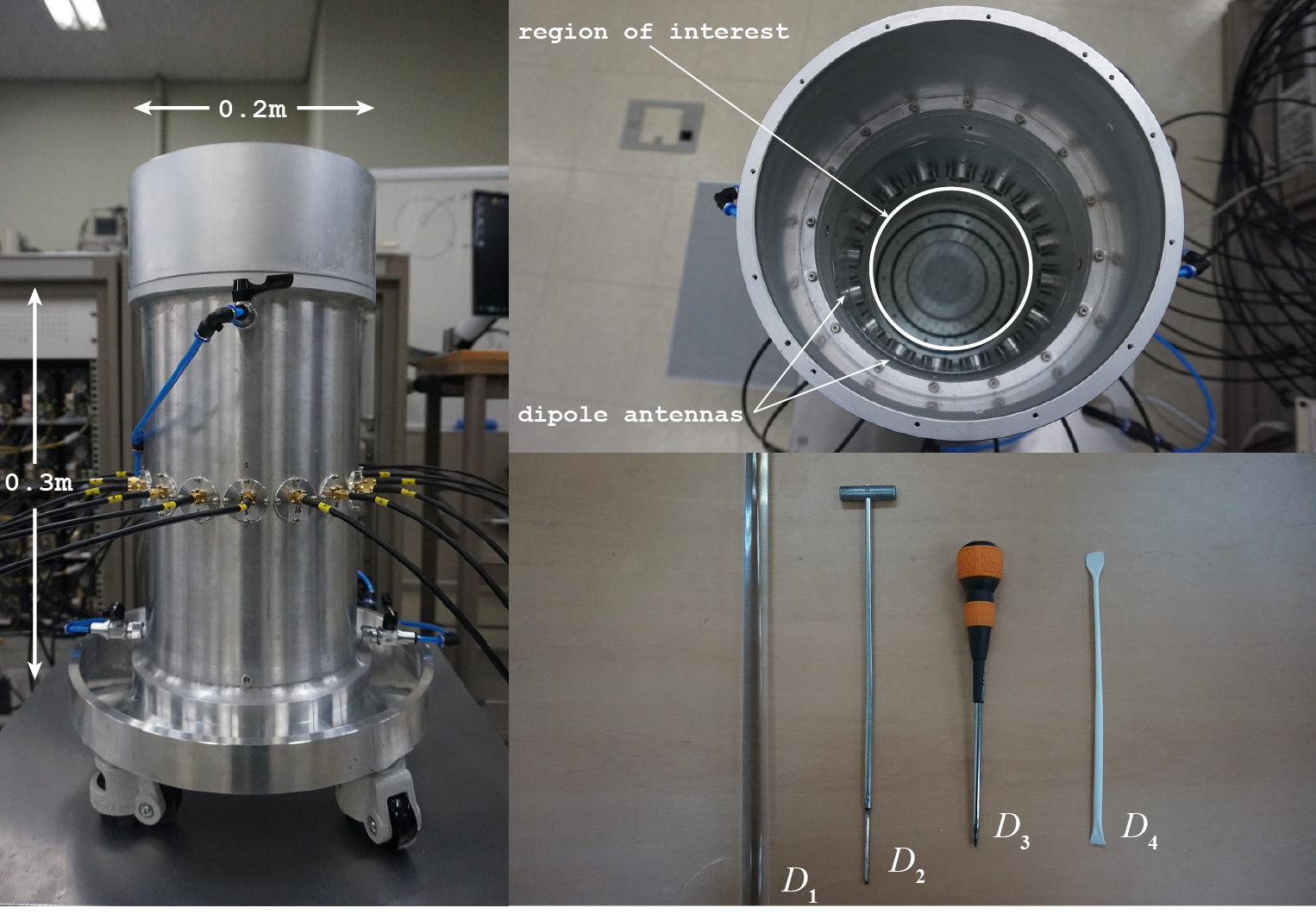}
\caption{\label{Configuration}Photos of microwave machine and objects $D_m$, $m=1,2,3,4$.}
\end{center}
\end{figure}

\begin{example}[Tracking of single moving object]\label{Ex1}
Figure \ref{Result-Single} shows the maps of $\mathfrak{F}(\mr,t)$ for a single, moving object $D_3(t)$. Although some artifacts are included on the map, it is possible to track the moving object. Note that at some time $t$ (e.g., $t=\SI{4.5}{\s}$, $\SI{5.0}{\s}$, $\SI{10.5}{\s}$, etc.), the artifacts were eliminated. We cannot explain the exact reason of this phenomenon, but we believe that at that moments, $\mathcal{E}(\kb|\mr-\mr'|,t)\approx0$ of \eqref{Structure}.
\end{example}

\begin{figure}[h]
\begin{center}
\includegraphics[width=\textwidth]{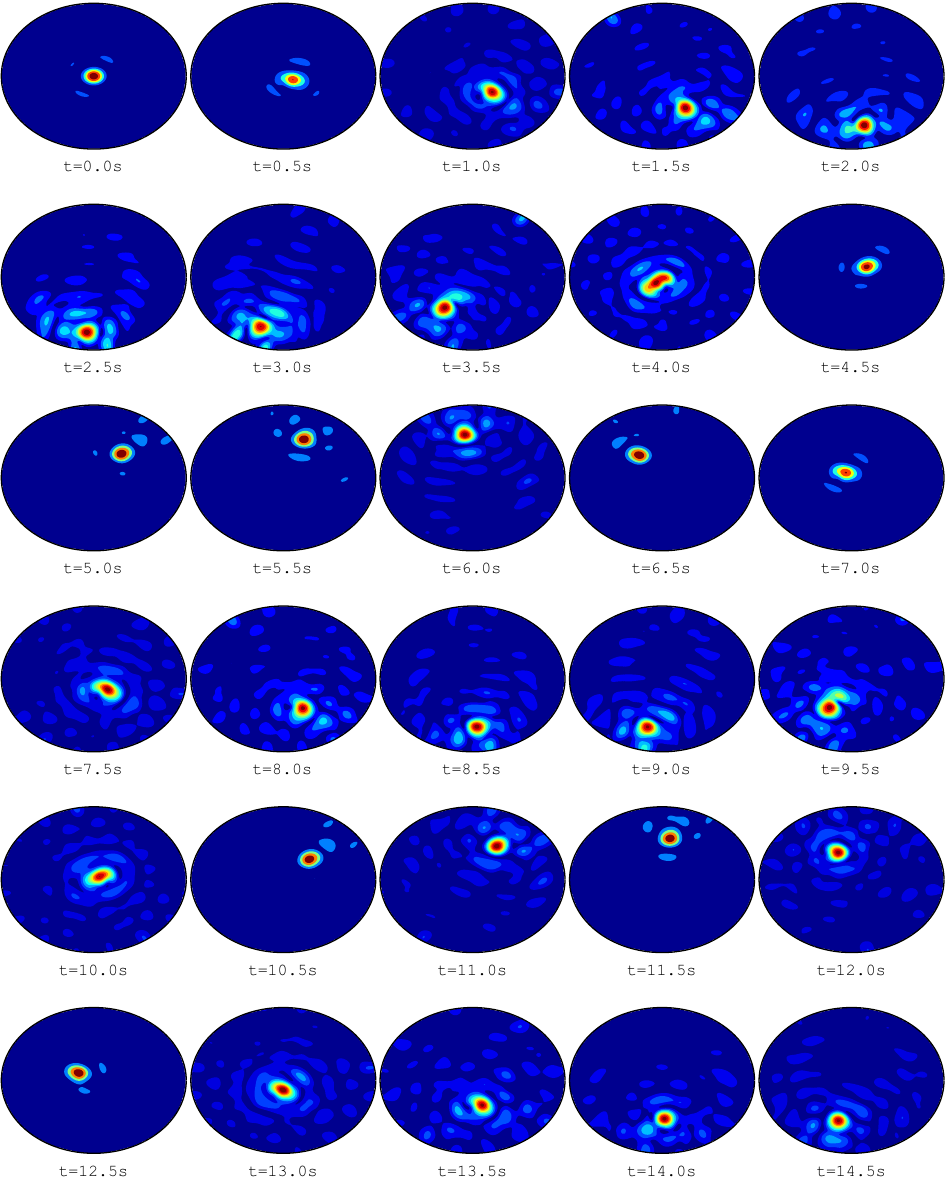}
\caption{\label{Result-Single}(Example \ref{Ex1}) Maps of $\mathfrak{F}(\mr,t)$ for moving object $D_3(t)$.}
\end{center}
\end{figure}

\begin{example}[Tracking of moving objects: same radii and material properties]\label{Ex2}
Figure \ref{Result-Multiple} shows the maps of $\mathfrak{F}(\mr,t)$ for two moving objects $D_2(t)$ and $D_3(t)$ with almost the same radii and material properties, i.e., $|\mathcal{C}(\mr_2,t)|\approx|\mathcal{C}(\mr_3,t)|$ for $\mr_2\in D_2(t)$ and $\mr_3\in D_3(t)$. Similar to the result in Example \ref{Ex1}, the track of the moving objects can be recognized even though some artifacts degrade the imaging quality. Fortunately, in contrast to Example \ref{Ex1}, the artifacts do not disturb the recognition of the moving objects but at certain instances $t=\SI{1.5}{\s}$ and $t=\SI{2.5}{\s}$, significantly large amounts of artifacts appear. It is interesting to observe that although the material properties and sizes of two objects are the same, the magnitudes are different at certain instances. For example, by comparing the maps of $\mathfrak{F}(\mr,t)$ at $t=\SI{11.5}{s}$ and $t=\SI{12.0}{s}$, we can observe that there is no significant change of locations, but the magnitudes changed. Unfortunately, the truth behind this phenomenon cannot be estimated at present, but the distance between the object and antennas may possibly be the reason, refer to Property \ref{Property5}.
\end{example}

\begin{figure}[h]
\begin{center}
\includegraphics[width=\textwidth]{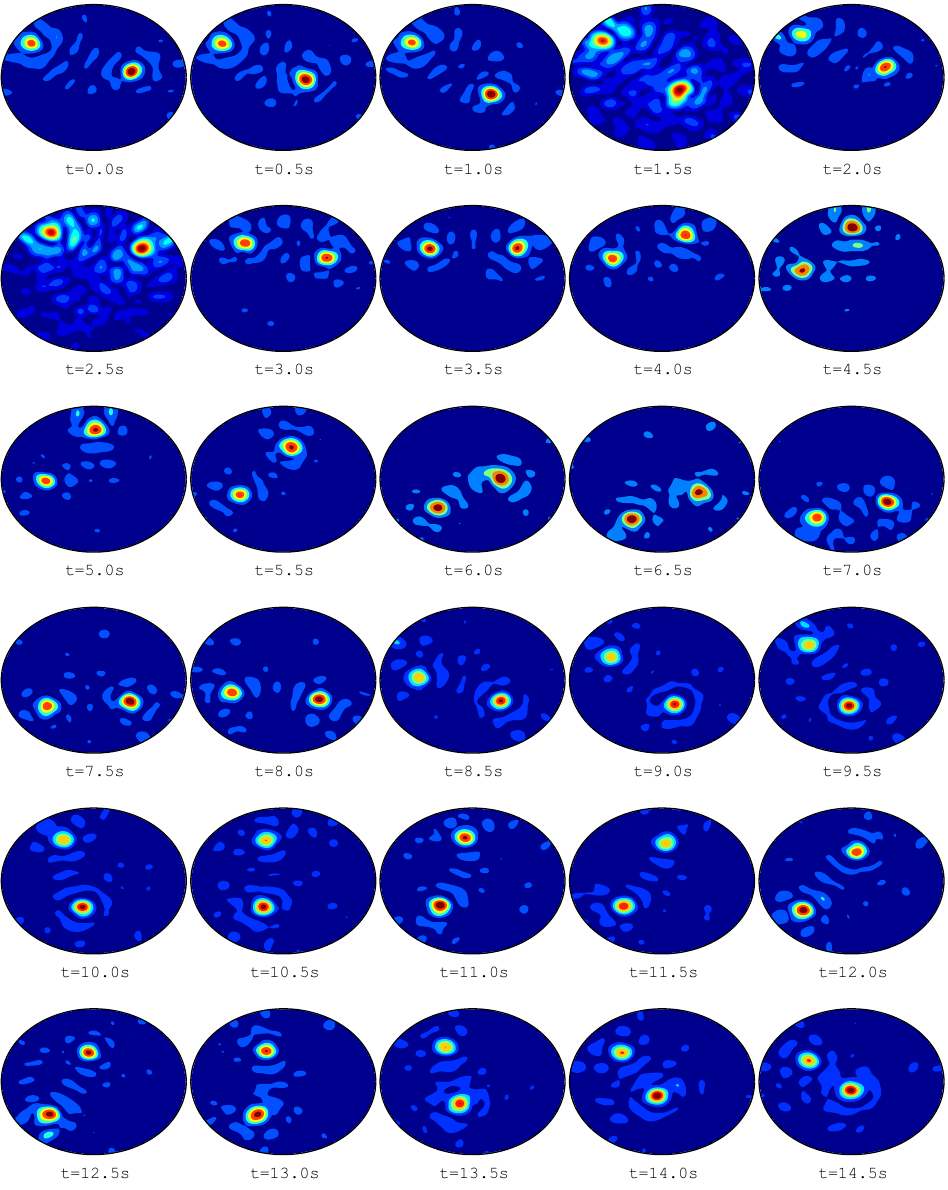}\caption{\label{Result-Multiple}(Example \ref{Ex2}) Maps of $\mathfrak{F}(\mr,t)$ for moving objects $D_2(t)$ and $D_3(t)$.}
\end{center}
\end{figure}

\begin{example}[Tracking of moving objects: different radii and material properties]
\label{Ex3}
Now, we consider the tracking of moving objects $D_1(t)$ and $D_2(t)$. On the basis of the result in Figure \ref{Result-Multiple-Different1}, we can observe that a peak of large magnitude appears when $\mr\in D_1(t)$ because the size of $D_1(t)$ is significantly larger than that of $D_2(t)$ i.e., $|\mathcal{C}(\mr_1,t)|>|\mathcal{C}(\mr_2,t)|$ for $\mr_1\in D_1(t)$ and $\mr_2\in D_2(t)$, refer to Property \ref{Property2}. Similar to the results in Example \ref{Ex2}, some artifacts are included on the map of $\mathfrak{F}(\mr,t)$; however, they were disappeared at certain instances, e.g., $t=\SI{1.0}{\s}$, $\SI{4.0}{\s}$, and $\SI{5.5}{\s}$.
\end{example}

\begin{figure}[h]
\begin{center}
\includegraphics[width=\textwidth]{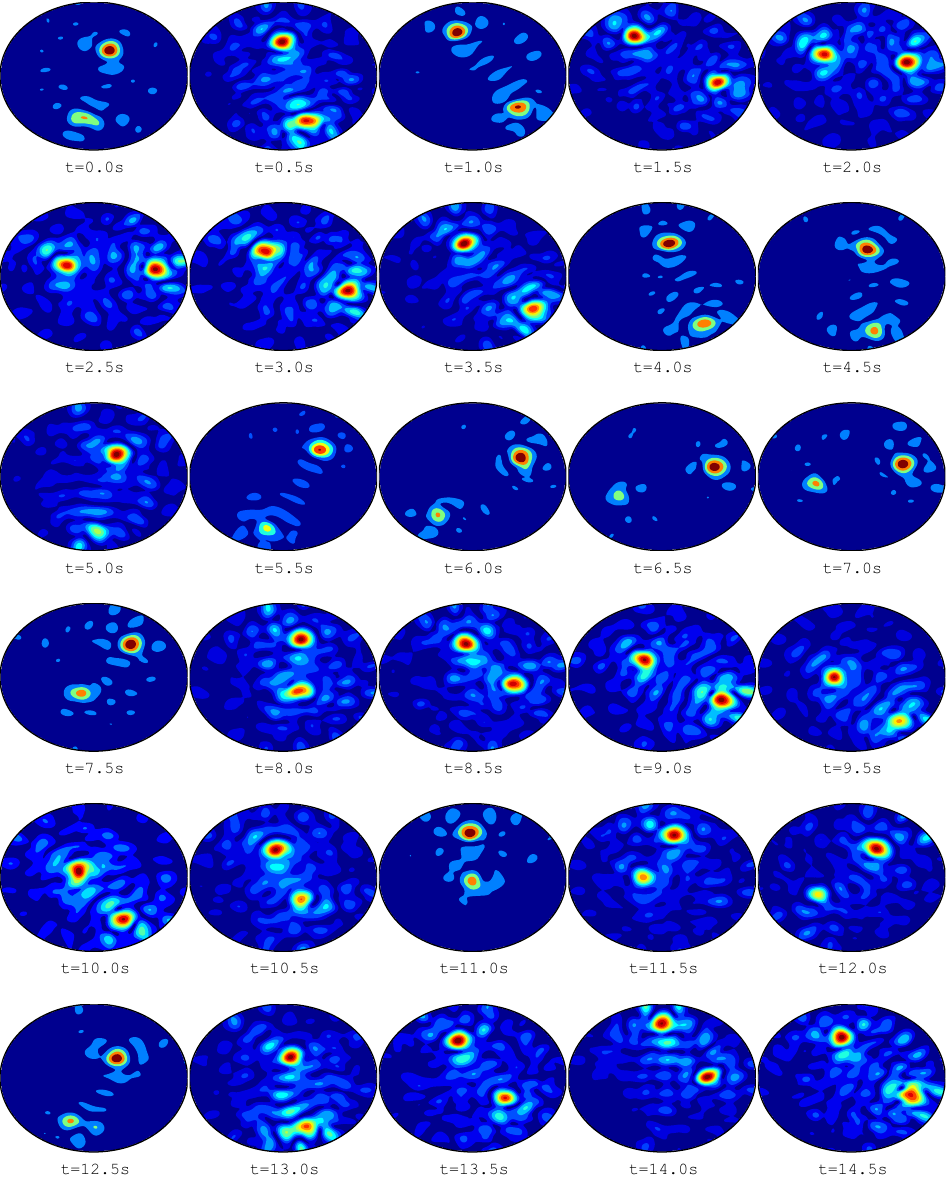}\caption{\label{Result-Multiple-Different1}(Example \ref{Ex3}) Maps of $\mathfrak{F}(\mr,t)$ for moving objects $D_1(t)$ and $D_2(t)$.}
\end{center}
\end{figure}

\begin{example}[Tracking of moving objects: same radii but different material properties]
\label{Ex4}
For the final example, we present the simulation result for moving objects $D_2(t)$ and $D_4(t)$ with the same size but different material properties such that $|\mathcal{C}(\mr_2,t)|\gg|\mathcal{C}(\mr_4,t)|$ for $\mr_2\in D_2(t)$ and $\mr_4\in D_4(t)$. Figure \ref{Result-Multiple-Different2} shows the maps of $\mathfrak{F}(\mr,t)$, and we can observe that it is impossible to recognize the movement of $D_4(t)$ because the value of $\mathfrak{F}(\mr,t)$ at $\mr\in D_4(t)$ is significantly smaller than that of $\mathfrak{F}(\mr,t)$ at $\mr\in D_2(t)$. It is interesting to observe that when an object moves quickly, few ghost replicas are exhibited on the map, e.g., two peaks of large magnitude appeared at $t=\SI{8.0}{\s}$ when $D_2(t)$ moved quickly.
\end{example}

\begin{figure}[h]
\begin{center}
\includegraphics[width=\textwidth]{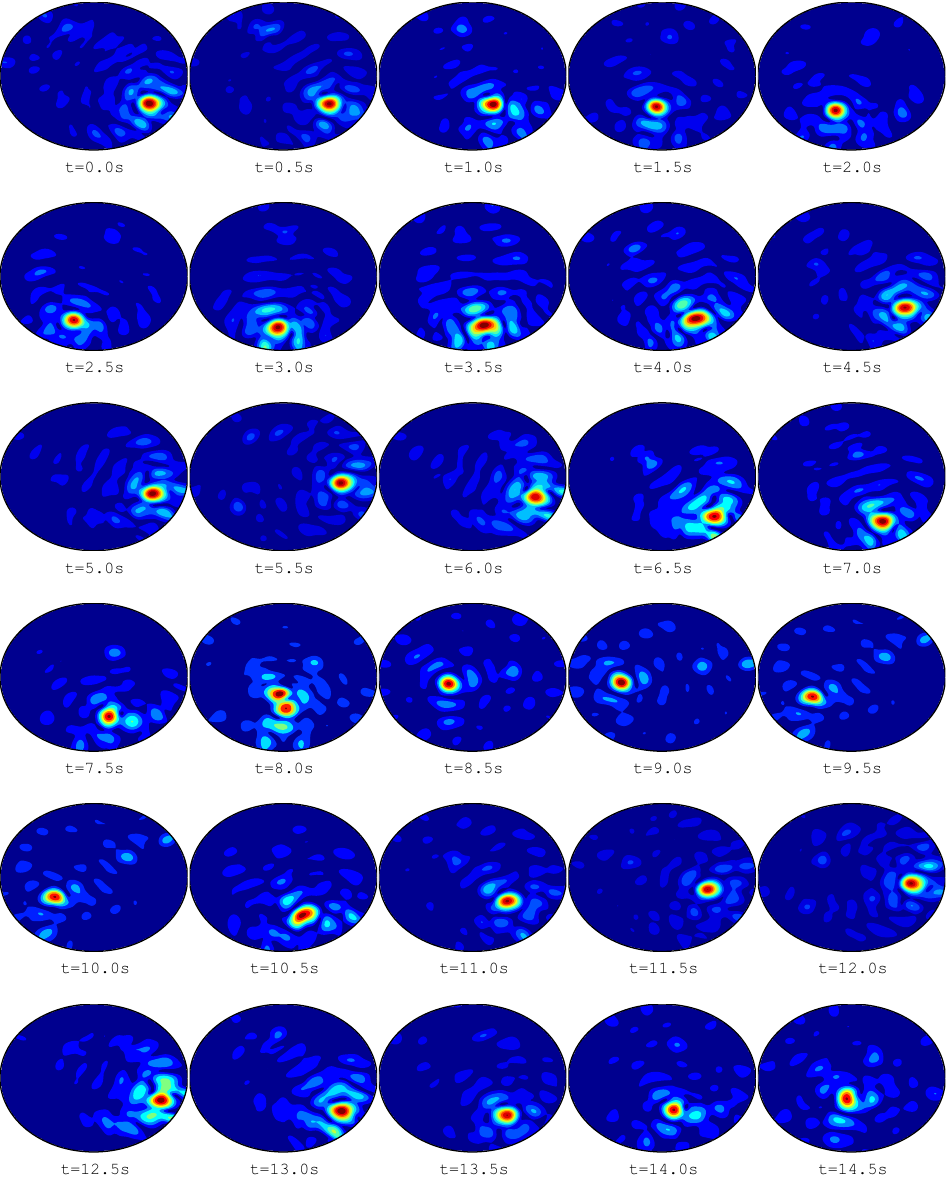}\caption{\label{Result-Multiple-Different2}(Example \ref{Ex4}) Maps of $\mathfrak{F}(\mr,t)$ for moving objects $D_2(t)$ and $D_4(t)$.}
\end{center}
\end{figure}

\section{Conclusion and perspectives}\label{sec:5}
Owing to the existence of small, moving objects and the non-iterative Kirchhoff migration technique in inverse scattering problems, we designed a real-time algorithm for imaging the moving objects on the basis of the representation formula of the scattering parameters. To examine the feasibility and explore some properties of the designed algorithm, we proved that the imaging function could be represented as the total and array configuration of the antennas and an infinite series of Bessel functions of integer order.

Numerical simulations were performed using experimental data generated by the ETRI to demonstrate that the Kirchhoff migration is very effective for a real-time detection of moving objects in microwave imaging. However, the algorithm's application is currently restricted to the detection of small objects; therefore, further applications to the detection of the movements of arbitrary shaped extended objects or the evolution of crack-like defects must be considered. Forthcoming studies will focus on designing appropriate imaging algorithms, performing related mathematical analyses, and conducting related simulations.

\section*{Acknowledgments}
This research was supported by the National Research Foundation of Korea (NRF) grant funded by the Korea government (MSIT) (NRF-2020R1A2C1A01005221) and the Soonchunhyang University research fund.

\end{document}